\documentclass[11pt,reqno]{amsart}    
\usepackage{geometry}                \usepackage{listings}

\usepackage{graphicx}
\graphicspath{{images/}}
\usepackage{blindtext}
\usepackage[normalem]{ulem}				
\usepackage{amssymb,hyperref}
\usepackage{amsfonts,amsmath,amscd,amsthm,enumitem}
\usepackage{thmtools, etoolbox, mathtools}
\usepackage{subfigure}
\usepackage{mathabx,geometry,caption}
\usepackage[dvipsnames]{xcolor}
\usepackage{epstopdf}
\usepackage{tikz-cd}
\usepackage[all]{xy}
\usepackage{adjustbox}
\usepackage{bbm}
\usepackage{faktor}
\usepackage{parskip}
\usepackage{xparse}
\usepackage[capitalize, noabbrev, nosort]{cleveref}
\usepackage{subfiles}

\usetikzlibrary{angles,quotes}

\usepackage{epigraph}
\let\oldmarginpar\marginpar
\renewcommand\marginpar[1]{\-\oldmarginpar[\raggedleft\footnotesize #1]%
	{\raggedright\footnotesize #1}}
\newtheorem{thm}{Theorem}[section] 

\newtheorem{lem}[thm]{Lemma}
\newtheorem{prop}[thm]{Proposition}

\theoremstyle{definition}

\numberwithin{claimcounter}{thm}
\newtheorem{claim}{Claim}
\newtheorem*{claim*}{Claim}
\newtheorem*{theorem*}{Theorem}

\newcommand{\dd}{\partial}

\newcommand{\Z}{\mathbb Z}

\newcommand{\Q}{\mathbb Q}

\newcommand{\N}{\mathbb N}

\newcommand{\R}{\mathbb R}

\usepackage{mathdots}
\usepackage{afterpage}
\makeatletter
\providecommand{\leftsquigarrow}{%
  \mathrel{\mathpalette\reflect@squig\relax}%
}
\newcommand{\reflect@squig}[2]{%
  \reflectbox{$\m@th#1\rightsquigarrow$}%
}
\makeatother

\makeatletter
\def\Ddots{\mathinner{\mkern1mu\raise\p@
\vbox{\kern7\p@\hbox{.}}\mkern2mu
\raise4\p@\hbox{.}\mkern2mu\raise7\p@\hbox{.}\mkern1mu}}
\makeatother

\def \horline [#1](#2,#3,#4){
    \draw [#1] (#2,#4) -- (#3,#4);
    \draw [fill=white] (#2,#4) circle [radius=0.1];
    \draw [fill=black] (#3,#4) circle [radius=0.1];
}

\def \crossing (#1,#2)(#3,#4){
\draw (#1,#2) -- (#3,#4);
\draw (#1,#4) -- (#3,#2);
}

\DeclareFontFamily{U}{mathb}{}
\DeclareFontShape{U}{mathb}{m}{n}{
  <-5.5> mathb5
  <5.5-6.5> mathb6
  <6.5-7.5> mathb7
  <7.5-8.5> mathb8
  <8.5-9.5> mathb9
  <9.5-11.5> mathb10
  <11.5-> mathbb12
}{}

\usetikzlibrary{decorations.markings, arrows}
\tikzset{tangent/.style={decoration={markings,mark=at position #1 with {
      \coordinate (tangent point-\pgfkeysvalueof{/pgf/decoration/mark info/sequence number}) at (0pt,0pt);
      \coordinate (tangent unit vector-\pgfkeysvalueof{/pgf/decoration/mark info/sequence number}) at (1,0pt);
      \coordinate (tangent orthogonal unit vector-\pgfkeysvalueof{/pgf/decoration/mark info/sequence number}) at (0pt,1);
      }},postaction=decorate},
    use tangent/.style={
        shift=(tangent point-#1),
        x=(tangent unit vector-#1),
        y=(tangent orthogonal unit vector-#1)
    },
    use tangent/.default=1
    }

\definecolor{codegreen}{rgb}{0,0.6,0}
\definecolor{codegray}{rgb}{0.5,0.5,0.5}
\definecolor{codepurple}{rgb}{0.58,0,0.82}
\definecolor{backcolour}{rgb}{0.95,0.95,0.92}

\lstdefinestyle{mystyle}{
    backgroundcolor=\color{backcolour},   
    commentstyle=\color{codegreen},
    keywordstyle=\color{magenta},
    numberstyle=\tiny\color{codegray},
    stringstyle=\color{codepurple},
    basicstyle=\ttfamily\footnotesize,
    breakatwhitespace=false,         
    breaklines=true,                 
    captionpos=b,                    
    keepspaces=true,                 
    numbers=left,                    
    numbersep=5pt,                  
    showspaces=false,                
    showstringspaces=false,
    showtabs=false,                  
    tabsize=2
}

\title{Boundedness criteria for real quivers of rank 3}
\author{Roger Casals and Kenton Ke}
\date{}

\begin{document}

\begin{abstract}
We study the boundedness of a mutation class for quivers with real weights. The main result is a characterization of bounded mutation classes for real quivers of rank 3.
\end{abstract}

\maketitle

\section{Introduction}

The object of this note is to introduce and study the notion of boundedness for mutation classes of quivers with real weights. In short, a quiver mutation class is said to be bounded if the coefficients of any of its quivers are uniformly bounded. This is a subtler notion for quivers with real weights, as opposed to quivers with integer weights, as we show that there exists bounded mutation classes with infinitely many quivers in them. Our main contribution is a characterization of rank 3 quivers with bounded mutation class, leading to the classification of such quivers with a criterion that can be readily verified. Two different proofs for the result are provided: the first proof is based on multivariate analysis and explicit bounds, while the second argument directly uses the geometric realization of quiver mutations by A.~Felikson and and P.~Tumarkin.

\subsection{Scientific context} Quivers and their mutations have acquired a prominent role in mathematics, especially since the introduction of cluster algebras by S.~Fomin and A.~Zelevinsky, cf.~\cite{FominZelevinsky_DoubleBruhat,FominZelevinsky_ClusterI,FominZelevinsky_ClusterII}. For a reference focused on the combinatorics of mutation, S.~Fomin presented a number of known results and open problems on quiver mutations in his talk at OPAC 2022, cf.~\cite{Fomin22_Opac}. As witnessed by the number of basic questions that remain open, it might be fair to state that the combinatorics governing quiver mutations remain rather mysterious. Some recent efforts to understand the combinatorics of quiver mutations have been fruitful, e.g.~ studying long mutation cycles, cf.~\cite{ervin2025mutationcyclesreddeningsequences,fomin2023longmutationcycles}, or constructing invariants of quiver mutation, cf.~ \cite{casals2023binaryinvariantmatrixmutation,fomin2024cyclicallyorderedquivers,neville2024mutationacyclicquiverstotallyproper,seven2024congruenceinvariantsmatrixmutation}.

\begin{center}
\begin{figure}[h!]
    \centering
	\includegraphics[width=\textwidth]{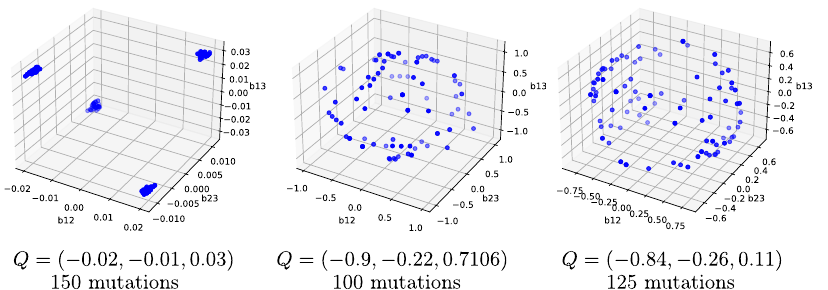}
		\caption{Depiction of the $(p,q,r)\in\R^3$ coordinates for three quivers $Q$ and a random sequence of mutations applied to them. In these cases, all three mutation classes $[Q]$ are bounded.  The mutation sequences have been generated at random and are provided in the appendix. In the main text the variables $p,q,r\in\R_{\geq0}$ are always non-negative but, to ease implementation, the code and the figures it produces have $(p,q,r)\in\R^3$. In this case, a negative sign can readily be transformed into a positive sign by reversing the corresponding arrow of the quiver: the arrow goes from vertex 1 to vertex 2, if $p$ positive, or from vertex 2 to vertex 1, if $p$ is negative.}\label{fig:Examples1}
\end{figure}
\end{center}

Thinking of mutations of a quiver as a discrete group action, it is reasonable to wonder about the dynamical properties of quiver mutation. In turn, it is often productive to study discrete dynamical systems in relation to real continuous dynamical systems. From that perspective, we are naturally led to study quivers with real weights and their orbits under mutations. We use the notation $Q=(p,q,r)$ for a quiver $Q$ with real weights $p,q,r\in\R_{\geq0}$. Here $p$ is the weight of the edge from vertex 1 to vertex 2. Similarly, the weight $q$ is for the arrow from vertex 2 to 3, and $r$ from vertex 1 to vertex 3.\footnote{For the purpose of Figures \ref{fig:Examples1} and \ref{fig:Examples2} only, we use the notation $Q=(p,q,r)$ for a quiver $Q$ with real weights $p,q,r\in\R$, where the negative sign indicates an arrow reversal. Throughout the main text $p,q,r\in\R_{\geq0}$, i.e.~ we only allow weights to be negative for the code and these two figures produced from it.} To depict this visually, suppose for instance that we have a quiver $Q=(p,q,r)$ with $p,q,r\in\R_{>0}$. Then the quiver $Q$ and its exchange matrix $B_Q$ are depicted as:

\vspace{-5mm}
\begin{equation*}
    \begin{tikzcd}[column sep=1cm, row sep=1cm]
& 2 \ar[rdd,"q"] \\
& \\
1 \ar[rr,"r"] \ar[ruu, "p"] && 3 \\
\end{tikzcd},
\qquad B_Q=\begin{pmatrix}
0 & p & r\\
-p & 0 & q\\
-r & -q & 0
\end{pmatrix}.
\end{equation*}
\vspace{-10mm}

\noindent To get a first sense, Figure \ref{fig:Examples1} depicts the quivers obtained by applying three random mutation sequences $(\mu_{i_\ell}\circ\ldots\circ\mu_{i_1})(Q)$ to three randomly chosen quivers $Q=(p,q,r)$ with real weights.\footnote{The specific mutation sequences $(i_1,\ldots,i_\ell)$ for each of these three quivers are written in \cref{ssec:explicit_seq}. We provide a code to generate such images in \cref{ssec:code}, with input the quiver $Q$ and the desired length of a mutation sequence.} Figure \ref{fig:Examples2} depicts three different random sequences of mutations applied to the same quiver $Q$. Note that, independently, the study of quivers with real weights has also gained recent attention due to their connection to the metric geometry of surfaces, cf.~e.g.~\cite{FT2017,felikson2022mutationfinitequiversrealweights,lampe2016approximateperiodicitysequencesattached} and references therein. In particular, mutation-finite quivers with real weights were beautifully classified in \cite[Theorem A]{felikson2022mutationfinitequiversrealweights} by A.~Felikson and P.~Tumarkin.

\begin{center}
	\begin{figure}[h!]
		\centering
        \includegraphics[width=\textwidth]{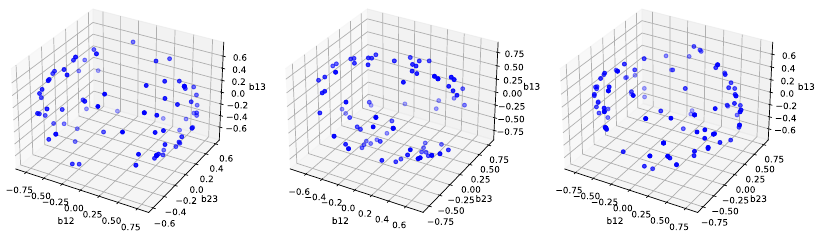}
		\caption{Depiction of the $(p,q,r)\in\R^3$ coordinates for quivers obtain by applying three different sequences to $Q=(-0.6,-0.43,0.567)$. The mutation class $[Q]$ is bounded. The mutation sequences have been generated at random, cf.~ \cref{sec:appendix}.}\label{fig:Examples2}
	\end{figure}
\end{center}

For quiver with real weights, the combinatorics of quiver mutation gains certain dynamical and geometric aspects that are not present for quivers with integer weights, see e.g.~\cite{Machacek02042024}. The focus of this note is to begin the study of one such property: the notion of a mutation class being bounded. Namely, understanding whether the coefficients of a given quiver remain uniformly bounded under an arbitrarily long sequence of mutations. For instance, we shall prove that the quivers $Q$ in Figure \ref{fig:Examples1} and \ref{fig:Examples2} have bounded mutation classes. For a quiver with integer weights, a quiver mutation class is bounded if and only if it is finite. For a quiver with real weights, such simple characterization fails: there are mutation classes that are infinite, i.e.~they contain infinitely many different quivers, and yet all quivers have their coefficients uniformly bounded.

A specific goal when studying a property of a quiver mutation class is to have a usable criterion to decide whether the mutation class of a given quiver possesses such property. In our case, the aim would be to be able to tell whether a given quiver with real weights has a bounded or unbounded mutation class. Since a mutation-finite quiver tautologically has a bounded mutation class, the focus is on deciding whether a mutation-infinite quiver has a bounded or unbounded mutation class. Our main result achieves this goal for quivers of rank 3.

\subsection{Main result} Let $Q$ be a quiver with real weights and $[Q]$ its mutation class. We refer to \cite[Chapter 2]{FWZ} for background on quiver mutations and \cite[Section 4]{lampe2016approximateperiodicitysequencesattached} for the case of real weights. The classification of finite-mutation type quivers with real weights is established in \cite[Thm.~5.9]{FT2017} for rank 3 and in \cite[Theorem A]{felikson2022mutationfinitequiversrealweights} for arbitrary rank.

\noindent Given a quiver $Q$, we denote by $\|Q\|:=\max\{|w|\in\R: w\mbox{ weight of }Q\}$ the maximum of the absolute values of the weights of $Q$. By definition, the norm of a quiver mutation class $[Q]$ is
$$\|[Q]\|:=\sup_{Q\in[Q]} \|Q\|.$$
By definition, $[Q]$ is said to be bounded if and only if $\|[Q]\|$ is finite. If $[Q]$ is mutation-finite, then $\|[Q]\|$ is finite and thus $[Q]$ is bounded. The converse holds if $Q$ has integer weights. In particular, there exist unbounded mutation-infinite classes $[Q]$. Interestingly, for quivers with real weights, we shall see that there are mutation-infinite classes $[Q]$ that are bounded.

\noindent Our exploration of boundedness focuses on the study of quivers of rank 3. For the remainder of the paper, we write $Q=(p,q,r)$ to indicate that the quiver $Q$ has weights $p,q,r\in\R_{\geq 0}$, where the cyclicity of the underlying directed graph $Q$ is provided as part of the input. The Markov constant $C(Q)\in\R$ of $Q$, as introduced in \cite[Section 1]{BBH2006}, is a mutation invariant of rank 3 quivers whose value governs important aspects of the behavior of $Q$ under mutation. To wit, \cite[Section 4]{FT2017} illustrates how the threshold $C(Q)=4$ marks a transition for the type of geometric realization of such mutation classes. The constant $C(Q)$ is defined as
\[C(Q):=
\begin{cases}
p^2+q^2+r^2-pqr & \mbox{if $Q=(p,q,r)$ is cyclic},\\
p^2+q^2+r^2+pqr & \mbox{if $Q=(p,q,r)$ is acyclic}.\\
\end{cases}\]

The main result of this article is the following characterization:
\newpage

\begin{thm}\label{thm:detbdd}
Let $Q$ be a quiver with real weights $(p,q,r)$ with $p\geq q\geq r\geq 0$. Then 
$$[Q]\mbox{ is bounded }\Longleftrightarrow p\leq2\mbox{ and }C(Q)\leq 4.$$
\end{thm}

\noindent We shall provide two different proofs for \cref{thm:detbdd}. In the first one, \Cref{thm:detbdd} will be deduced from the following result, which is the main analytical part of the article:

\begin{prop}\label{thm:main}
Let $Q$ be a rank $3$ quiver with real weights, different from the Markov Quiver, and $[Q]$ its mutation class. Then
$$[Q]\mbox{ is bounded }\Longleftrightarrow \mbox{$[Q]$ is mutation acyclic and $C(Q)\leq 4$.}$$

In fact, more explicitly:
\begin{itemize}
    \item[(i)] If $[Q]$ is bounded then $\|[Q]\|\leq\sqrt{C(Q)}$.
    \item[(ii)] If $[Q]$ is unbounded then starting at any cyclic $Q\in[Q]$ and iteratively mutating at the vertex opposite to the smallest weight produces quivers of arbitrarily large and non-decreasing norm.
\end{itemize}
\end{prop}

\noindent \Cref{thm:main} provides a complete characterization of bounded mutation classes but it is nevertheless not optimal. Indeed, given a quiver $Q$ it still requires being able to decide whether $[Q]$ is mutation acyclic. In contrast, the strength of \Cref{thm:detbdd} is that it allows us to determine boundedness of a mutation class $[Q]$ for any given quiver $Q$ by just using the given quiver itself, and not its mutations.\\

\noindent Part of the non-trivial content of \Cref{thm:main} is the assertion that a mutation-infinite quiver $Q$ with $C(Q)>4$ has an unbounded mutation class $[Q]$. Correspondingly, in \Cref{thm:detbdd}, that a quiver $Q$ with $C(Q)>4$ and real weights $p,q,r\leq 2$ has an unbounded mutation class $[Q]$. Namely, we show that it does not matter how small the given weights $p,q,r\in\R_{\geq0}$ are: if $C(Q)>4$, then $[Q]$ is unbounded. In line with \cref{thm:main}.(ii), we prove these facts by exhibiting explicit mutation sequences that arbitrarily increase quiver weights under these hypotheses. See \cref{lem:mu_*} below and the inequality in \cref{eq:Markov_inequality} for an instance of the lower bounds that appear.\\

In this first proof of \cref{thm:detbdd}, the techniques and arguments are rather elementary, the bounds are explicit, and the proof can be followed without substantial prior knowledge on the subject, which is a strength. At the same time, the arguments are quite lengthy and do not necessarily provide conceptual reasons of why the statements are true. Fortunately, the referee suggested an alternative strategy to prove \cref{thm:detbdd}, based on \cite{FT2017}. In \cite{FT2017}, A.~Felikson and P.~Tumarkin provided a way to model all mutation classes of rank $3$ quivers geometrically, based on spherical, Euclidean and hyperbolic geometries, depending on the case. We will show that \cref{thm:detbdd} can also be proven using these geometric models from \cite{FT2017}, as the referee suggested to us. This second proof helps gaining a conceptual understanding for why \cref{thm:detbdd} holds and, once one assumes the techniques \cite{FT2017} as prior knowledge, it provides a beautifully visual and compelling argument for this boundedness behavior.\\

\noindent {\bf Acknowledgments}. We are truly grateful to the referee for their detailed reading of the manuscript and their suggestions. Specifically, much of the alternative proof in \cref{sec:alternative_proof_mainresults} was sketched by the referee themselves and it has been a joy to learn from and implement it. We are thankful that it can be included in this manuscript, but it should be clear that the merit for the alternative proof belongs to the referee. We are thankful to Scott Neville and Melissa Sherman-Bennett for their helpful comments on a first draft of the manuscript and useful discussions. R.C.~is supported by the National Science Foundation under grants DMS-2505760 and DMS-1942363, and a UC Davis College of L\&S Dean's Fellowship.

\section{First proof of the main results}\label{sec:proof_mainresults}

The first proof of \cref{thm:main} and \cref{thm:detbdd} are structured as follows. \cref{ssec:proof_char} establishes in \cref{thm:2iffbdd} a first characterization of bounded mutation classes. Such characterization depends on the (non)existence of a cyclic quiver with particular weights. \cref{ssec:proof_lemmas} proves two lemmas on the possible weights of rank 3 quivers depending on cyclicity and the Markov constant. \cref{ssec:proof_main} then uses the results from \cref{ssec:proof_char} and \cref{ssec:proof_lemmas} to prove \cref{thm:main}. \cref{ssec:proof_cor} then proves \cref{thm:detbdd} from \cref{thm:main}. Throughout this section, $Q$ is a connected rank 3 quiver with nonnegative real weights.

\subsection{A preliminary characterization}\label{ssec:proof_char} The goal of this subsection is to prove \cref{thm:2iffbdd}. The argument uses the following assertion:

\begin{lem}\label{lem:exists cyclic}
$[Q]$ contains a cyclic quiver.
\end{lem}
\begin{proof}
Since $Q$ is connected, any quiver in $[Q]$ must have at least two non-zero weights. Suppose that a quiver $Q\in[Q]$ is given with exactly two non-zero weights, so $Q$ is one of the following quivers, up to relabeling of vertices and reversing orientation:
    \begin{center}
    \begin{tikzcd}[column sep=1cm, row sep=1cm]
& 2 \ar[rdd] \ar[ldd] \\
& \\
1 && 3 \\
\end{tikzcd}
\hspace{5mm}
    \begin{tikzcd}[column sep=1cm, row sep=1cm]
& 2 \ar[rdd] \\
& \\
1 \ar[ruu] && 3 \\
\end{tikzcd}
\hspace{5mm}
\begin{tikzcd}[column sep=1cm, row sep=1cm]
& 2 \\
& \\
1 \ar[ruu] && 3 \ar[luu] \\
\end{tikzcd}
\end{center}
\vspace{-10mm}
But these three quivers are mutation equivalent, via the mutations $\mu_1$ then $\mu_3$ from left to right, so we can assume without loss of generality that the acyclic quiver $Q$ is the middle quiver. In that case, mutation at vertex $2$ yields the cyclic quiver:
\begin{center}
    \begin{tikzcd}[column sep=1cm, row sep=1cm]
& 2 \ar[rdd, "q"]\\
&\\
1 \ar[ruu, "p"] && 3\\
\end{tikzcd}
$\xleftrightarrow{\text{\hspace{5mm}$\mu_2$}\hspace{5mm}}$
\begin{tikzcd}[column sep=1cm, row sep=1cm]
& 2 \ar[ldd,"p"] \\
& \\
1 \ar[rr,"pq"] && 3 \ar[luu, "q"] \\
\end{tikzcd}
\end{center}
\vspace{-10mm}
If $Q$ had been an acyclic quiver with three nonzero weights then, up to relabeling vertices and reversing orientation, $Q$ is mutation equivalent to 
\vspace{-5mm}
\begin{center}
    \begin{tikzcd}[column sep=1cm, row sep=1cm]
& 2 \ar[rdd,"q"] \\
& \\
1 \ar[rr,"r"] \ar[ruu, "p"] && 3 \\
\end{tikzcd}
\end{center}
\vspace{-10mm}
which is itself mutation equivalent, via $\mu_2$, to the following cyclic quiver:
\begin{center}
    \begin{tikzcd}[column sep=1cm, row sep=1cm]
& 2 \ar[ldd,"p"] \\
& \\
1 \ar[rr,"pq+r"] && 3 \ar[luu, "q"] \\
\end{tikzcd}
\end{center}
\vspace{-10mm}
Either case, we obtain a cyclic quiver in $[Q]$, as required.
\end{proof}

\noindent Here follows a preliminary classification of bounded mutation classes in rank 3:

\begin{prop}\label{thm:2iffbdd}
Let $Q$ be a rank $3$ quiver with real weights. Then
\begin{equation*}
[Q]\mbox{ is unbounded }\Longleftrightarrow \exists Q'\in [Q] \mbox{ with } Q' \mbox{ cyclic and one weight larger than }2.
\end{equation*}
\end{prop}

\begin{proof}
For $(\Longrightarrow)$ we proceed as follows. Since $[Q]$ is unbounded, there exists a quiver $Q'\in [Q]$ with one weight greater than $2$. If $Q'$ is cyclic, then we are done. Suppose thus that $Q'$ is acyclic. The same proof of \cref{lem:exists cyclic} implies that there must be a cyclic quiver $Q''$ with one weight greater than 2. Indeed, the following two cases cover all such possibilities. If $Q'$ is the following quiver on the left
    \vspace{-5mm}
    \begin{center}
    \begin{tikzcd}[column sep=1cm, row sep=1cm]
Q' & 2 \ar[rdd, "q"]\\
&\\
1 \ar[ruu, "p>2"] && 3\\
\end{tikzcd}
$\xleftrightarrow{\text{\hspace{5mm}$\mu_2$}\hspace{5mm}}$
\begin{tikzcd}[column sep=1cm, row sep=1cm]
Q'' & 2 \ar[ldd,"p>2"] \\
& \\
1 \ar[rr,"pq"] && 3 \ar[luu, "q"] \\
\end{tikzcd}
\end{center}
\vspace{-10mm}
then $Q''$ is chosen to be the quiver on its right. If $Q'$ is instead the left quiver in
\vspace{-5mm}
\begin{center}
    \begin{tikzcd}[column sep=1cm, row sep=1cm]
Q' & 2 \ar[rdd, "q"]\\
&\\
1 \ar[ruu, "p>2"] \ar[rr,"r"] && 3\\
\end{tikzcd}
$\xleftrightarrow{\text{\hspace{5mm}$\mu_2$}\hspace{5mm}}$
\begin{tikzcd}[column sep=1cm, row sep=1cm]
Q'' & 2 \ar[ldd,"p>2"] \\
& \\
1 \ar[rr,"pq+r"]&& 3 \ar[luu, "q"] \\
\end{tikzcd}
\end{center}
\vspace{-10mm}
then $Q''$ is chosen to be the corresponding quiver to its right. This concludes $(\Longrightarrow)$.\\

Let us prove $(\Longleftarrow)$. For that, we assume that $Q'=(p_0,q_0,r_0)\in [Q]$ is the given cyclic quiver with weights satisfying $0< r_0\leq q_0\leq p_0$ and $p_0>2$. Without loss of generality, we can assume that $Q'$ is of the following form:
\begin{center}
\begin{tikzcd}[column sep=1cm, row sep=1cm]
Q' & 2 \ar[rdd,"q_0"] \\
& \\
1 \ar[ruu,"p_0"]&& 3 \ar[ll, "r_0"] \\
\end{tikzcd}
\end{center}
\vspace{-10mm}
The goal is to show that $[Q]$ is unbounded. We prove that by iteratively mutating at vertices $1$ and $2$, which we momentarily show forces the weights of the quivers in the sequence to arbitrarily increase. Specifically, let us
define the numbers $q_i,r_i\in\R_{\geq0}$ recursively by
$$q_i:=p_0r_i-q_{i-1},\quad r_i:=p_0q_{i-1}-r_{i-1},\quad i>0.$$
We claim that these numbers are the weights of the quivers appearing in the mutation sequence starting at $Q'$ and alternately mutating at vertices $2$ and $1$. The beginning of the sequence is:
\begin{center}
\begin{tikzcd}[column sep=1cm, row sep=1cm]
Q' & 2 \ar[rdd,"q_0"] \\
& \\
1 \ar[ruu,"p_0"]&& 3 \ar[ll, "r_0"] \\
\end{tikzcd}
$\xleftrightarrow{\text{\hspace{5mm}$\mu_2$}\hspace{5mm}}$
    \begin{tikzcd}[column sep=1cm, row sep=1cm]
& 2 \ar[ldd, "p_0"]\\
&\\
1 \ar[rr, "r_1=p_0q_0-r_0"] && 3 \ar[luu, "q_0"] \\
\end{tikzcd}
$\xleftrightarrow{\text{\hspace{5mm}$\mu_1$}\hspace{5mm}}$
\begin{tikzcd}[column sep=1cm, row sep=1cm]
& 2 \ar[rdd,"q_1=p_0r_1-q_0"] \\
& \\
1 \ar[ruu,"p_0"]&& 3 \ar[ll, "r_1"] \\
\end{tikzcd}
\end{center}
\vspace{-10mm}
To verify this claim, we need to check that the arrows $q_i,r_i$ are indeed going in the direction that retains cyclicity, i.e. $q_i,r_i>0$ for all $i>0$. In addition, we also assert that the inequalities $q_i>r_i>q_{i-1}$ hold for all $i>0$. We prove such inequalities
\begin{equation}\label{eq:inequality1}
q_i>r_i>q_{i-1}>0
\end{equation}
by induction on $i\in\N$. The base case is $i=1$: since $r_0\leq q_0$ and $p_0>2$, we must have
$$r_1=p_0q_0-r_0>2q_0-q_0=q_0>0,\mbox{ and }q_1=p_0r_1-q_0>2r_1-r_1=r_1>0.$$
For the inductive step, we assume that $q_k>r_k>q_{k-1}>0$ for some $k\in \N$. Then it follows that
$$r_{k+1}=p_0q_k-r_k>2q_k-q_k=q_k>0,\mbox{ and }q_{k+1}=p_0r_{k+1}-q_k>2r_{k+1}-r_{k+1}=r_{k+1}>0,$$
which proves \cref{eq:inequality1}. To conclude that $[Q]$ is unbounded, it suffices to show that these weights $q_i,r_i$ increase without bound as we iterate the mutations at $2$ and $1$, i.e.~as $i\to\infty$. This is a consequence of the following:
\begin{claim}\label{claim1}
        $\displaystyle \lim_{i\rightarrow{\infty}}q_i=\infty$.
    \end{claim}
\begin{proof}[Proof of \cref{claim1}] Consider the quantity $\mathbf{p}:=p_0^2-p_0-1$. The hypothesis $p_0>2$ implies $\mathbf{p}>1$ and thus $\displaystyle \lim_{i\rightarrow \infty}\mathbf{p}^i=\infty$. We claim that the sequence of weights $(q_i)$ satisfies \begin{equation}\label{eq:inequality2}
q_i\geq\mathbf{p}^iq_0,\quad\mbox{for all }i>0.
\end{equation}
\Cref {eq:inequality2} can be established by induction on $i$, as follows. The base case is $i=1$, and since $r_0\leq q_0$, we have 
    \[
    q_1=p_0r_1-q_0=p_0^2q_0-p_0r_0-q_0\geq p_0^2q_0-p_0q_0-q_0=(p_0^2-p_0-1)q_0=\mathbf{p}^1q_0.
    \]
For the inductive step, we assume that $q_k\geq \mathbf{p}^kq_0$ for some $k\in \N$. Given that we have shown previously that $r_k<q_k$, it follows that
    \[
    q_{k+1}=p_0r_{k+1}-q_k=p_0^2q_k-p_0r_k-q_k>p_0^2q_k-p_0q_k-q_k=(p_0^2-p_0-1)q_k\geq\mathbf{p}\cdot \mathbf{p}^kq_0=\mathbf{p}^{k+1}q_0,
    \]
    which implies \cref{eq:inequality2}. Since $\displaystyle \lim_{i\rightarrow \infty}\mathbf{p}^i=\infty$, it follows that
    $$\displaystyle \lim_{i\rightarrow \infty}q_i>\displaystyle q_0\cdot \lim_{i\rightarrow \infty}\mathbf{p}^i=q_0\cdot \infty=\infty,$$
    as required.
\end{proof}
\noindent The fact that $[Q]$ is unbounded now follows from Claim \ref{claim1}, as we constructed a sequence of quivers in $[Q]$ whose coefficients increase arbitrarily as we iteratively mutate at the vertices $2$ and $1$.
\end{proof}

\subsection{Two quick lemmas}\label{ssec:proof_lemmas} A caveat of \cref{thm:2iffbdd} is that it is in generally challenging to determine the existence or non-existence of a cyclic quiver with one weight larger than $2$ in a given mutation class $[Q]$. This makes \cref{thm:2iffbdd} difficult to use in practice. Hence, we are motivated to further explore the behavior of weights of rank 3 quivers, which will lead to \cref{thm:main}, improving \cref{thm:2iffbdd}. The two necessary lemmas that we shall use in the proof \cref{thm:main} read as follows:

\begin{lem}\label{lem:sqrt2}
Let $Q=(p,q,r)$ be a cyclic quiver with real weights $0<r\leq q\leq p$.\\
If $C(p,q,r)>4$, then $p>\sqrt{2}$.
\end{lem}

\begin{proof}
Let us argue by contradiction, assuming that $p\leq \sqrt{2}$. Since $Q$ is cyclic, $C(Q)$ reads
$$C(p,q,r)=p^2+q^2+r^2-pqr.$$
Considered as a real smooth function of $p,q,r$, we have $\dd_p C=2p-qr>0$ since $q,r\leq p$ and we are assuming $p\leq\sqrt{2}$. Thus, $C(p,q,r)$ is a strictly increasing function of $p$ and it attains its maximum when $p$ is maximized, that is, when $p=\sqrt{2}$. If we write $f(q,r):=C(\sqrt{2},q,r)=2+q^2+r^2-\sqrt{2}qr$, this implies
$$C(p,q,r)\leq f(q,r).$$
Note that $\dd_qf=2q-\sqrt{2}r>0$ because $2>\sqrt{2}$ and $q\geq r$. Thus $f(q,r)$ is a strictly increasing function of $q$ and it attains its maximum when $q$ is maximized. Since we have $q\leq p\leq\sqrt{2}$, we can write $g(r):=C(\sqrt{2},\sqrt{2},r)=2+2+r^2-2r=4+r^2-2r$ and the following inequality will hold
$$C(p,q,r)\leq f(q,r)\leq g(r).$$
Then $g(r)$ is maximized in the same way: $\dd_rg=2r-2>0$ if $r>1$, and $\dd_rg<0$ if $r<1$. Hence, $g(r)$ is increasing for $r>1$ and decreasing for $r<1$. Since $0<r\leq q\leq p\leq \sqrt{2}$, in the interval $[0,\sqrt{2}]$ the function $g(r)$ attains its global maximum at the boundary, i.e.~either at $r=\sqrt{2}$ or $r=0$. Note that $g(\sqrt{2})=4+2-2\sqrt{2}<4=g(0)$, thus we have that $g(r)<g(0)=4$ in the interval $(0,\sqrt{2}]$. Hence, 
    \[
    C(p,q,r)\leq f(q,r)\leq g(r)<4.
    \]
    This contradicts the hypothesis $C(Q)>4$, and thus we must have had that $p>\sqrt{2}$, as required.
\end{proof}

Lemma \ref{lem:sqrt2} gives a lower bound for $p$ whenever $C(p,q,r)>4$: this has an important role in the proof of the upcoming \cref{lem:C(Q)/4q}, and ultimately \cref{thm:main}. Appropriately used, this next lemma will provide a way to construct a sequence of mutations with strictly increasing weights:

\begin{lem}\label{lem:C(Q)/4q}
Let $Q=(p,q,r)$ be a cyclic quiver with weights $0<r\leq q\leq p<2$. If $C(p,q,r)>4$, then
$$\displaystyle pq-r>\frac{C(Q)}{4}q.$$
\end{lem}

\begin{proof}
Consider the smooth function
$$f:\R^3_{\geq0}\longrightarrow\R,\quad f(p,q,r):=(pq-r)-\frac{C(Q)}{4}q,$$
which we need to show is positive. Let $g(p,q,r):=\dd_pf=q-q\cdot \frac14(2p-qr)$. Since $p<2$, we have $2p-qr<4-qr<4$ and hence $g(p,q,r)>0$ is positive and $f$ is increasing as a function of $p$. By Lemma \ref{lem:sqrt2}, the hypothesis $C(p,q,r)>4$ implies the inequality $\sqrt{2}< p<2$. Altogether, for any $p\in(\sqrt{2},2)$, $f(p,q,r)>f(\sqrt{2},q,r)$ and so it suffices to show $f(\sqrt{2},q,r)>0$.
By Lemma \ref{lem:sqrt2} again, if $p=\sqrt{2}$ then $C(p,q,r)\leq 4$, and thus we inspect this extremal case of $f(\sqrt{2},q,r)$, for $q,r$ satisfying $0\leq r\leq q\leq \sqrt{2}$, and $C(\sqrt{2},q,r)=4$. Since $C(\sqrt{2},q,r)=2+q^2+r^2-\sqrt{2}qr$, we have that
$$\dd_qC(\sqrt{2},q,r)=2q-\sqrt{2}r>0.$$
Thus, $C(\sqrt{2},q,r)$ is increasing as a function of $q$, hence it attains its maximum at $q=\sqrt{2}$ with value $C(\sqrt{2},\sqrt{2},r)=4+r^2-2r=4+r(r-2)$. Since $r\in(0,\sqrt{2}]$, $r(r-2)<0$ and so $C(p,q,r)<C(\sqrt{2},\sqrt{2},0)=4$. In particular, the only remaining extremal point $(\sqrt{2},q,r)$ satisfying $C(\sqrt{2},q,r)=4$ is $(\sqrt{2},\sqrt{2},0)$, and the value of $f$ is
$$f(\sqrt{2},\sqrt{2},0)=2-\sqrt{2}\cdot \frac44=2-\sqrt{2}>0.$$
In conclusion, under the hypothesis $C(p,q,r)>4$, we indeed have $f(p,q,r)>f(\sqrt{2},\sqrt{2},0)>0$.
\color{black}
\end{proof}

\subsection{Proof of \Cref{thm:main}}\label{ssec:proof_main} Let us first prove the implication $(\Longleftarrow)$. If $Q$ is the Markov Quiver, then $[Q]$ is tautologically bounded. Therefore, we assume $[Q]$ is mutation acyclic and $C(Q)\leq 4$. We will now show that $[Q]$ is bounded by $\sqrt{C(Q)}$, as stated in \cref{thm:main}.(i).

Let $(\alpha,\beta,\gamma)\in [Q]$ be an acyclic quiver. Then $\alpha^2+\beta^2+\gamma^2+\alpha\beta\gamma=C(\alpha,\beta,\gamma)=C(Q)$, and it follows that $\alpha^2,\beta^2,\gamma^2\leq C(Q)$, hence $\alpha,\beta,\gamma\leq \sqrt{C(Q)}$. Thus the upper bound $\sqrt{C(Q)}$ holds for acyclic quivers with $C(Q)\leq 4$. Let $Q=(p,q,r)\in [Q]$ be cyclic, and let $d(p,q,r)$ be the minimum number of mutations required to transform $Q=(p,q,r)$ into an acyclic quiver. Such $d$ is finite because of the hypothesis that $[Q]$ is mutation acyclic. We show that $\sqrt{C(Q)}$ is an upper bound by induction on $d$.

\indent The base case is $d=1$, i.e.~ $Q=(p,q,r)$ is one mutation away from an acyclic quiver $Q'=(\alpha,\beta,\gamma)$. Since $Q'$ is acyclic, we have already shown $\alpha,\beta,\gamma\leq \sqrt{C(Q)}$. Noting that a single mutation to $(p,q,r)$ preserves at least two weights, we may assume without loss of generality that $q,r\leq \sqrt{C(Q)}$, and we need to show $p\leq \sqrt{C(Q)}$. We argue by contradiction, so we assume that $p>\sqrt{C(Q)}$. We have $C(Q)=p^2+q^2+r^2-pqr$: we now want to get a lower bound for $C(p,q,r)$, and ultimately for $p$. Let $f(p):=p^2+q^2+r^2-pqr$ be considered as a smooth real function of $p$, then
\begin{equation}\label{eq:f_increasing}
f'(p)=2p-qr>0 \iff p>\frac{qr}{2}.
\end{equation}
By assumption, $C(Q)\leq 4$ and thus $\sqrt{C(Q)}\leq 2$. Since we also have $q,r\leq \sqrt{C(Q)}$, we obtain
$$\frac{qr}{2}\leq \frac{C(Q)}{2}\leq\frac{C(Q)}{\sqrt{C(Q)}}=\sqrt{C(Q)}<p.$$
This implies \ref{eq:f_increasing}, that is, $f(p)$ is strictly increasing and thus $f(p)$ is minimized when $p$ is minimized. We are assuming that $p>\sqrt{C(Q)}$ and therefore we have the sequence of implications:
    \begin{align*}
        &f(p)=p^2+q^2+r^2-pqr>C(Q)+q^2+r^2-\sqrt{C(Q)}qr\\
        &\implies 0>q^2+r^2-2qr+2qr-\sqrt{C(Q)}qr\\
        &\implies 0>(q-r)^2+\left(2-\sqrt{C(Q)}\right)qr.
    \end{align*}
Since $(q-r)^2\geq 0$, and $\sqrt{C(Q)}\leq 2$ implies $\left(2-\sqrt{C(Q)}\right)qr\geq 0$, the last inequality implies
$$0>(q-r)^2+\left(2-\sqrt{C(Q)}\right)qr\geq0$$
which is a contradiction. Hence, $p\leq \sqrt{C(Q)}$ as required, concluding the base case.\\

\noindent For the induction step, we assume the statement to be true for any $d$ with $1\leq d\leq n\in \N$, and let $Q'=(p',q',r')\in [Q]$ be a cyclic quiver with $d(p',q',r')=n+1$. Let $(\alpha,\beta,\gamma)\in [Q]$ be a cyclic quiver with $d(\alpha,\beta,\gamma)=n$ that is one mutation from $(p',q',r')$. By the induction hypothesis, we have $\alpha,\beta,\gamma\leq \sqrt{C(Q)}$. Since one mutation preserves at least two weights, we assume without loss of generality that $q',r'\leq \sqrt{C(Q)}$. Then, applying the same argument as in the base case, we readily conclude $p'\leq \sqrt{C(Q)}$.
    
\noindent To summarize, we have shown that all acyclic and cyclic quivers in $[Q]$ are bounded by $\sqrt{C(Q)}$. This concludes the proof of \cref{thm:main}.(i) and the implication $(\Longleftarrow)$.\\ 

Let us prove the implication $(\Longrightarrow)$, which we show by contradiction, and the statement in \cref{thm:main}.(ii). For that, we consider two cases, the most interesting being the second case:

    \textbf{Case 1.} Suppose that {\it $[Q]$ is bounded, mutation cyclic, and $Q$ is not the Markov Quiver.} Since $[Q]$ is bounded and mutation cyclic, \cref{thm:2iffbdd} implies that any quiver $Q'=(p,q,r)\in [Q]$ with $0<r\leq q\leq p$ must have $p,q,r\leq 2$. If $p=q=r=2$, then $Q'$ is the Markov quiver, which is a contradiction. Thus, we must have $r<2$. By \cite[Lemma 3.3]{FT2017}, it follows that $[Q]$ is mutation acyclic, which is a contradiction. Therefore, we conclude that if $[Q]$ is bounded and $Q$ is not the Markov Quiver, then $[Q]$ is mutation acyclic, establishing the mutation acyclic part of \cref{thm:main}.(2) in this case.

    \textbf{Case 2.} Suppose that {\it $[Q]$ is bounded and $C(Q)>4$.} Note that we do not need the assumption that $Q$ is not the Markov Quiver, since the Markov Quiver has Markov Constant equal to $4$. The intuitive idea is that given any cyclic quiver $Q\in[Q]$, we will manage to use the inequality $C(Q)>4$ to produce an arbitrarily long sequence of mutations such that the weights increase arbitrarily. Such mutation sequence depends on the starting quiver $Q\in[Q]$ to which we apply the procedure, which makes this argument finer than that in the proof of \cref{thm:2iffbdd}. This is done via the following lemma.

\setcounter{claim}{0}
\begin{lem}\label{lem:mu_*}
Assume $Q=(p_0,q_0,r_0)$ is a cyclic quiver with $0<r_0\leq q_0\leq p_0\leq 2$ and $C(Q)>4$. Define $\mu_*(Q)=(p_1,q_1,r_1)$ to be the mutation of $Q$ at the vertex opposite the smallest weight, then relabeling so that $0<r_1\leq q_1\leq p_1\leq 2$. Recursively define
$$\mu_*^i(Q)=(p_i,q_i,r_i):=\mu_*(\mu_*^{i-1}(Q))$$
with the same properties that $0<r_i\leq q_i\leq p_i\leq 2$, and $p_i\geq p_{i-1}q_{i-1}-r_{i-1}$. Then for all $i\geq0$, the following inequality holds:
\begin{equation}\label{eq:Markov_inequality}
\displaystyle p_iq_i-r_i>\left(\frac{C(Q)}{4}\right)^{\left\lfloor\frac{i+2}{2}\right\rfloor}q_0.
\end{equation}
\end{lem}

\begin{proof}
For the sake of clarity, we first describe this mutation sequence in more detail. Without loss of generality, we may assume $Q$ is the following quiver:
\begin{center}
    \begin{tikzcd}[column sep=1cm, row sep=1cm]
Q & 2 \ar[rdd, "q_0"]\\
&\\
1 \ar[ruu, "p_0"] && 3\ar[ll,"r_0"]\\
\end{tikzcd}
\end{center}
\vspace{-10mm}
Since the vertex opposite the smallest weight $r_0$ is vertex $2$, we have $\mu_*(Q)=\mu_2(p_0,q_0,r_0):=(p_1,q_1,r_1)$ By mutating $Q$ at vertex $2$ we obtain the following quiver on the right:
    \begin{center}
    \begin{tikzcd}[column sep=1cm, row sep=1cm]
Q & 2 \ar[rdd, "q_0"]\\
&\\
1 \ar[ruu, "p_0"] && 3\ar[ll,"r_0"]\\
\end{tikzcd}
$\xleftrightarrow{\text{\hspace{5mm}$\mu_2$}\hspace{5mm}}$
\begin{tikzcd}[column sep=1cm, row sep=1cm]
\mu_2(Q)\hspace{-5mm} & 2 \ar[ldd,"p_0"] \\
& \\
1 \ar[rr,"p_0q_0-r_0"]&& 3 \ar[luu, "q_0"] \\
\end{tikzcd}
\end{center}
\vspace{-10mm}
and so $(p_1,q_1,r_1)=(p_0,p_0q_0-r_0,q_0)$ if $p_0\geq p_0q_0-r_0$, and $(p_1,q_1,r_1)=(p_0q_0-r_0,p_0,q_0)$ if $p_0\leq p_0q_0-r_0$. Note that by \cref{lem:C(Q)/4q}, the weight $p_0q_0-r_0$ will never be the smallest weight $r_1$.
Now, we prove the desired inequality by induction on $i$. For the base case, consider $i=0$ and $i=1$. For $i=0$, Lemma \ref{lem:C(Q)/4q} implies that
$$p_0q_0-r_0>\frac{C(Q)}{4}\cdot q_0=\left(\frac{C(Q)}{4}\right)^{\left\lfloor\frac{0+2}{2}\right\rfloor} q_0.$$
\noindent For $i=1$, there are two cases to consider:
    \begin{enumerate}
        \item If $\mu_*^1(Q)=(p_0,p_0q_0-r_0,q_0)=(p_1,q_1,r_1)$, then Lemma \ref{lem:C(Q)/4q} implies
        \[p_1q_1-r_1>\frac{C(Q)}{4}q_1=\frac{C(Q)}{4}(p_0q_0-r_0)>\left(\frac{C(Q)}{4}\right)^2q_0>\left(\frac{C(Q)}{4}\right)^{\left\lfloor\frac{1+2}{2}\right\rfloor}q_0.\]
        \vspace{-3mm}
        \item If $\mu_*^1(Q)=(p_0q_0-r_0,p_0,q_0)=(p_1,q_1,r_1)$, then we directly have
        \[p_1q_1-r_1>\frac{C(Q)}{4}p_0>\frac{C(Q)}{4}q_0=\left(\frac{C(Q)}{4}\right)^{\left\lfloor\frac{1+2}{2}\right\rfloor}q_0.\]
    \end{enumerate}
\noindent This concludes the base cases $i=0,1$. For the induction step, we assume that
$$p_iq_i-r_i>\left(\frac{C(Q)}{4}\right)^{\left\lfloor\frac{i+2}{2}\right\rfloor}q_0$$
holds for all $0\leq i\leq n$ and want to show it holds for $i=n+1$. The induction hypothesis for $i=n$ is
$$p_nq_n-r_n>\left(\frac{C(Q)}{4}\right)^{\left\lfloor\frac{n+2}{2}\right\rfloor}q_0$$
with the inequalities $p_n\geq p_{n-1}q_{n-1}-r_{n-1}>\left(\frac{C(Q)}{4}\right)^{\left\lfloor\frac{n+1}{2}\right\rfloor}q_0$. There are two cases for the next quiver $\mu_*^{n+1}(Q)$ in the mutation sequence:
        \begin{enumerate}
            \item If $\mu_*^{n+1}(Q)=(p_n,p_nq_n-r_n,q_n)=(p_{n+1},q_{n+1},r_{n+1})$, then Lemma \ref{lem:C(Q)/4q} applies to give
            \begin{align*}
            \hspace{-5mm}p_{n+1}q_{n+1}-r_{n+1}&>\frac{C(Q)}{4}q_{n+1}\\
            &=\frac{C(Q)}{4}(p_nq_n-r_n)\\
            &>\left(\frac{C(Q)}{4}\right)^{\left\lfloor\frac{n+2}{2}\right\rfloor+1}q_0\\
            &=\left(\frac{C(Q)}{4}\right)^{\left\lfloor\frac{n+4}{2}\right\rfloor}q_0\\
            &\geq \left(\frac{C(Q)}{4}\right)^{\left\lfloor\frac{(n+1)+2}{2}\right\rfloor}q_0.
            \end{align*}
            \item If $\mu_*^{n+1}(Q)=(p_nq_n-r_n,p_n,q_n)=(p_{n+1},q_{n+1},r_{n+1})$, then we use Lemma \ref{lem:C(Q)/4q} as follows:
            \begin{align*}
            p_{n+1}q_{n+1}-r_{n+1}&>\frac{C(Q)}{4}q_{n+1}\\
            &= \frac{C(Q)}{4}p_n\\
            &\geq\frac{C(Q)}{4}(p_{n-1}q_{n-1}-r_{n-1})\\
            &>\left(\frac{C(Q)}{4}\right)^{\left\lfloor\frac{n+1}{2}\right\rfloor+1}q_0\\
            &\geq \left(\frac{C(Q)}{4}\right)^{\left\lfloor\frac{(n+1)+2}{2}\right\rfloor}q_0,
            \end{align*}
        \end{enumerate}
        which indeed yields \cref{eq:Markov_inequality} for $i=n+1$.
    \end{proof}
    
\noindent Since $C(Q)>4$ and thus $\frac{C(Q)}{4}>1$, \cref{lem:mu_*} implies 
$$\displaystyle \lim_{i\rightarrow \infty}p_iq_i-r_i>\lim_{i\rightarrow \infty}\left(\frac{C(Q)}{4}\right)^{\left\lfloor\frac{i+2}{2}\right\rfloor}q_0=\infty.$$
By construction, we have $$\displaystyle \lim_{i\rightarrow \infty}\text{max}(\mu_*^i(Q))\geq\displaystyle \lim_{i\rightarrow \infty}p_iq_i-r_i=\infty,$$
hence, $[Q]$ is unbounded, which is a contradiction. Therefore, it must have been that $C(Q)<4$ if $[Q]$ were bounded and infinite, which shows that the implication $(\Longrightarrow)$ holds in \cref{thm:main}. By construction, the statement in \cref{thm:main}.(ii) follows from \cref{lem:mu_*}.\hfill$\Box$

\subsection{Proof of \cref{thm:detbdd}}\label{ssec:proof_cor}

Let us prove $(\Longrightarrow)$ by establishing the contrapositive. First, let us assume $p>2$ and we want to conclude $[Q]$ is unbounded. If $Q$ is cyclic, then  \cref{thm:2iffbdd} implies $[Q]$ is unbounded. We thus assume the given quiver $Q$ is acyclic. Following the proof of \cref{lem:exists cyclic}, there exists a mutation $\mu$ so that $\mu(Q)$ is cyclic and $\text{Max}(\mu (Q))\geq \text{Max}(Q)$. Therefore $\mu(Q)$ is a cyclic quiver with $p>2$ and \cref{thm:2iffbdd} implies $[Q]=[\mu(Q)]$ is unbounded.

Second, let us assume $C(Q)>4$ and we want to conclude $[Q]$ is unbounded. By contradiction, suppose $[Q]$ is bounded. By \cref{thm:main}, $Q$ is either the Markov Quiver, or $[Q]$ is mutation acyclic and $C(Q)\leq 4$. However, the Markov Quiver has Markov Constant equal to $4$, hence in either case we get a contradiction to the assumption that $C(Q)>4$, thus $[Q]$ is not bounded.

\noindent Let us prove $(\Longleftarrow)$ directly from \cref{thm:main}. We are assuming $p\leq 2$ and $C(Q)\leq4$ and want to conclude $[Q]$ is bounded. First, we consider the case where the inequality $p<2$ is strict. By \cite[Lemma 3.3]{FT2017}, such mutation class $[Q]$ must be mutation acyclic, and since we already have $C(Q)\leq4$, \cref{thm:main} implies that $[Q]$ is bounded. Second, consider the case of equality $p=2$. If $Q$ is acyclic, then $[Q]$ is mutation acyclic and \cref{thm:main} again proves $[Q]$ is bounded. Therefore, we assume $Q$ is cyclic. In this case, if $p=q=r=2$, then $Q$ is the Markov quiver, and hence $[Q]$ is bounded. So it remains to check the cases where $r<2$. By \cite[Corollary 4.8]{FT2017}, $[Q]$ must then be mutation acyclic and \cref{thm:main} implies that $[Q]$ is bounded.\hfill$\Box$

\section{Alternative Proof of Main Result}\label{sec:alternative_proof_mainresults}
Let us now present an alternative proof of \cref{thm:detbdd}, using the geometric models from \cite{FT2017}. The reader is referred to the well-written and inspiring account \cite{FT2017} for the necessary details. Since the geometric model realizing quiver mutations depends on whether we have a mutation-acyclic class or a mutation-cyclic class, we must consider these two cases separately.\\

\textbf{Mutation-cyclic case}. Suppose that the given quiver $Q$ is mutation-cyclic. By \cite[Theorem 3.6]{FT2017}, a geometric realization of such a quiver is provided by three points $a_1,a_2,a_3\in\mathbb{H}^2$. The real weight between vertices $i,j$ of the quiver is given by $2\cosh(d(a_i,a_j))$, where $d$ is the hyperbolic distance in $\mathbb{H}^2$, cf.~\cite[Section 3.1.1]{FT2017}. In this case, since $\cosh(x)\geq 1$ for all $x\in\R$, the weights $p,q,r\in\R_{\geq 0}$ of a mutation-cyclic quiver $Q$ are all greater than or equal to $2$.

\begin{center}
	\begin{figure}[h]
		\centering
        \includegraphics[scale=0.8]{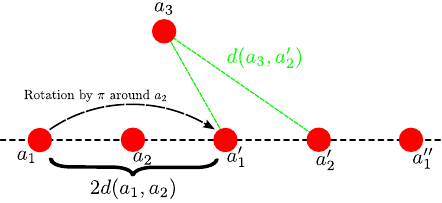}
		\caption{Depiction of the geometric realization of a mutation-cyclic rank $3$ quiver, where we are applying alternating mutations of $a_1,a_2$, realized by $\pi$-rotations, in order to translate them far away from $a_3$. In this figure, the points $a_1',a_2',a_1'',\dots$ are obtained by $\pi$-reflecting along $a_1$ and $a_2$ and iterating.}\label{fig:mutationcyclic}
	\end{figure}
\end{center}\vspace{-5mm}

\noindent We can consider two cases now:

\begin{itemize}
    \item[(i)] If the three points $a_1,a_2,a_3\in\mathbb{H}^2$ coincide, then the hyperbolic distance between any of the pairs is $0$. Therefore, we obtain $p=q=r=2\cosh(0)=2$, which is the Markov quiver. The Markov quiver satisfies $C(Q)=4$ and is bounded, in accordance with the statement of \cref{thm:detbdd}.

    \item[(ii)] Else, at least one of the points $a_i$ is distinct from the others: we can assume $a_1\neq a_2$ without loss of generality. In this case the distance $d(a_1,a_2)>0$ is positive, which corresponds to a mutation-cyclic quiver $Q=(p,q,r)$ with weight $p>2$. By \cite[Section 3.2]{FT2017}, quiver mutation is geometrically realized by a rotation by $\pi$ of one of the points $a_i$ with respect to another point, while keeping the other two points fixed. Thus, if we alternately rotate by $\pi$ along $a_1$ and then along $a_2$, then the images of $a_1,a_2$ will shift by $2d(a_1,a_2)$ each time. See \cref{fig:mutationcyclic} for a depiction. In particular, both points $a_1,a_2$ can move as far away from $a_3$ as needed.\footnote{This argument is independent of the location of $a_3$, e.g.~ it could be that $a_3=a_1$ or $a_2$ and the argument applies.} In consequence, both distances $d(a_1,a_3)$ and $d(a_2,a_3)$ approach infinity through this process, and we conclude that the mutation class $[Q]$ is unbounded. 
\end{itemize}
\noindent This concludes the proof of \cref{thm:detbdd} for the mutation-cyclic cases.\\

\textbf{Mutation-acyclic case}. Suppose that the quiver $Q$ is mutation-acyclic. By \cite[Theorem 4.4 \& Remark 4.8]{FT2017}, the geometric realization of such a quiver is a collection of three lines $l_i,l_j,l_k$ in one of $S^2,\mathbb{E}^2$, or $\mathbb{H}^2$, depending on whether $C(Q)<4$, $C(Q)=4$, or $C(Q)>4$ respectively, cf.~\cite[Figure 3]{FT2017}. Note that the lines may be intersecting, parallel (or coinciding), or disjoint, in the case of $\mathbb{H}^2$. By \cite[Section 2.2.1]{FT2017}, weights are computed as follows:
\begin{itemize}
    \item[(i)] If two lines $l_i,l_j$ intersect, then they form an angle $\alpha_{ij}$, and the corresponding real weight in the quiver is given by $|2\cos(\alpha_{ij})|$.
    \item[(ii)]  If $l_i,l_j$ coincide or are parallel (in $\mathbb{E}^2$), then we view them as intersecting with angle $0$, hence the corresponding weight in the quiver is $|2\cos(0)|=2$. If $l_i,l_j$ are disjoint in $\mathbb{H}^2$, then the corresponding weight of the quiver is $2\cosh(d(l_i,l_j))$, where $d$ measures the hyperbolic distance between two disjoint lines.
\end{itemize}
\noindent In this geometric realization, quiver mutation corresponds to reflecting one of the lines $l_i$ across another line $l_j$ while keeping the third line $l_k$ fixed. Since the geometric realization depends on the value of $C(Q)$, we must consider the following three cases:

\begin{enumerate}
\item {\bf Case $C(Q)<4$}. Then the geometric realization is in the 2-sphere $S^2$ in this case. In $S^2$, any two lines either intersect or coincide. Thus any corresponding quiver would have weights given by $|2\cos(\alpha_{ij})|$, and hence we have the upper bound $p,q,r\leq 2$. It follows that $[Q]$ is bounded by $2$.\\

\item {\bf Case $C(Q)=4$}. Then the geometric realization is in the Euclidean plane $\mathbb{E}^2$. In $\mathbb{E}^2$, any two lines either intersect, coincide, or are parallel. Thus the same argument for the case in $S^2$ applies, and we conclude that $[Q]$ is bounded by $2$.\\

\item {\bf Case $C(Q)>4$}. In this case, the geometric realization is in $\mathbb{H}^2$ and we shall show that this implies that the quiver mutation class $[Q]$ is unbounded. By construction, arbitrarily large weights must correspond to large distances between two disjoint lines in $\mathbb{H}^2$. Therefore, the goal is to prove that we can always apply a sequence of reflections to create arbitrarily large distances between disjoint lines in these hyperbolic realizations. The argument in this hyperbolic case is a bit more involved, and we establish it using the following steps:

    \begin{itemize}
        \item[(a)] First, we consider the case where the line $l_j$ is disjoint from the line $l_k$. Regardless of the location of the third line $l_i$, we can alternate mutations by first reflecting $l_j$ across $l_k$, thus obtaining a new line $l_j'$ while fixing $l_i$, and then reflecting $l_k$ across $l_j'$. Note that since $l_j$ and $l_k$ were disjoint to begin with, their images are also disjoint from each other. In particular, these reflections are hyperbolic translations along the common perpendicular to $l_j$ and $l_k$. Reiterating this process and renaming the correspondingly new lines $l_j',l_k'$, we conclude that both distances $d(l_j',l_i), d(l_k',l_i)$ increase arbitrarily, cf.~\Cref{fig:H2disjointquiver}. Therefore, it follows that two of the weights of the corresponding quiver are $p,q\in\{2\cosh(d(l_j',l_i)),2\cosh(d(l_k',l_i))\}$, both of which blow-up to infinity under this process, and hence $[Q]$ must be an unbounded mutation class.
        
        \noindent Note that this argument in $(a)$ shows that in order to show that a mutation class is unbounded, it suffices to find a geometric realization where two of the lines are disjoint.
        \begin{center}
	\begin{figure}[h]
		\centering
        \includegraphics[scale=0.8]{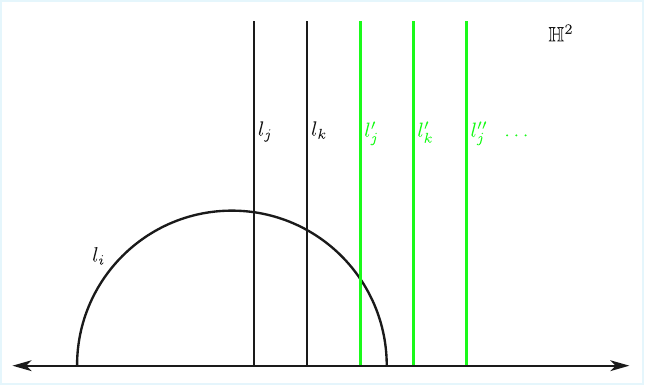}
		\caption{Depiction of the case where $l_j$ is disjoint from $l_k$, as in (a) above. The lines $l_j',l_k',l_j'',l_k'',\cdots$ are obtained by applying alternating mutations: these reflect $l_j$, $l_k$ across each other's images. Therefore, the distances $d(l_k',l_i),d(l_k'',l_i),d(l_k''',l_i),\dots$ increase arbitrarily, and same with $d(l_j',l_i),d(l_j'',l_i),d(l_j''',l_i),\dots$ forming an unbounded sequence. Note that $l_j,l_k$ need not both be vertical, but it is perhaps easier to visualize the case when $l_j,l_k$ are vertical, as depicted.}\label{fig:H2disjointquiver}
	\end{figure}
\end{center}
        \item[(b)] Second, let us consider the case where the three lines in $\mathbb{H}^2$ bound a triangle. Just for the time being, let us also assume that there is a pair of lines $l_j,l_k$ that form an intersection angle $\theta_i$ that is not a rational multiple of $\pi$. (The general case will be argued below.) Then, using the upper-half plane model for $\mathbb{H}^2$, we may apply a M\"obius transformation so that the line $l_i$ opposite to the angle $\theta_i$ is vertical.\footnote{These transformations are conformal, and thus the angles of the triangle do not change. Hence the resulting quiver remains the same.} In particular, the intersection $x=l_j\cap l_k$ of $l_j$ and $l_k$ is not on the line $l_i$. See \cref{fig:H2quiver} for an illustration. In consequence, when we apply a sequence of alternating mutations, given by reflecting $l_j$ across $l_k$, the angle $\theta_i$ of the intersection between the images of $l_j,l_k$ does not change, and so the entire angle from the initial $l_j$ to these images is precisely $n\theta_i$ for some $n\in\mathbb{Z}$. Since we assumed that $\theta_i$ is an irrational multiple of $\pi$, the set $\{n\theta_i \pmod{2\pi}|n\in \mathbb{Z}\}$ is dense in $[0,2\pi]$. This implies that there is some $l_j'$ through the intersection point $x$ that is arbitrarily close to being vertical. In particular, $l_i$ and such an $l_j'$ are disjoint, and we have reduced to case (a), resolved above.
      
    \begin{center}
	\begin{figure}[h]
		\centering
        \includegraphics[scale=0.8]{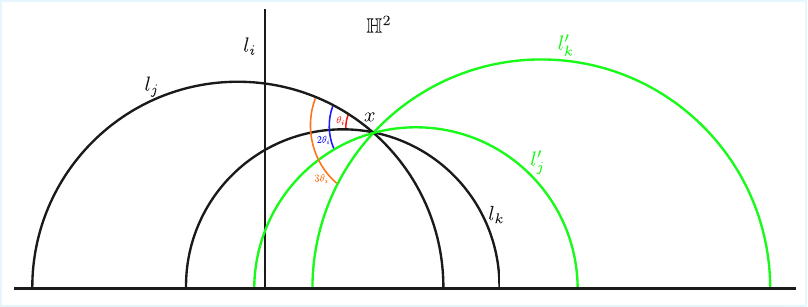}
		\caption{A depiction of the geometric realization of a rank $3$ mutation-acyclic quiver $Q$ with $C(Q)>4$ in the upper-half plane model $\mathbb{H}^2$. In this case, the mutated lines $l_j',l_k'$ are obtained by applying alternating reflections of $l_j,l_k$ along each other.}\label{fig:H2quiver}
	\end{figure}
\end{center}
\vspace{-5mm}
    
    \item[(c)] Finally, it remains to consider the case where the three lines in $\mathbb{H}^2$ bound a triangle and all three angles $\theta_i,\theta_j,\theta_k$ of the triangle are rational multiples of $\pi$. In this case, the weights of the corresponding quiver are given by $p,q,r\in \{|2\cos(\theta_i)|, |2\cos(\theta_j)|,|2\cos(\theta_k)|\}$ and, by $(b)$ above, it suffices to show that there is some quiver in the mutation class where at least one of the angles in its geometric representation is not a rational multiple of $\pi$.
    
    Depending on whether the quiver is acyclic or cyclic, a mutation could give us a new weight either of the form
    $$pq+r=2\cos(\theta_i)2\cos(\theta_j)+2\cos(\theta_k)$$
    or $$|pq-r|=||2\cos(\theta_i)2\cos(\theta_j)|-|2\cos(\theta_k)||,$$
    with possibly interchanging $\theta_i,\theta_j,\theta_k$ among each other. Note that the absolute values are placed here because it is proven in \cite{FT2017} that all acyclic quivers correspond to a geometric realization where all the angles are acute, hence $\cos(\theta)>0$, whereas all cyclic quivers correspond to a geometric realization where not all the angles are acute, and thus there is the possibility for $\cos(\theta)<0$. These identities above can be manipulated by first applying the product to sum identity for the cosine: $$\cos\phi_1\cos\phi_2=\frac12 (\cos(\phi_1-\phi_2)+\cos(\phi_1+\phi_2)).$$
    Indeed, in the first case, applying the identity to the right hand side reads:
    $$2\cos(\theta_i)2\cos(\theta_j)+2\cos(\theta_k)=2\cos(\theta_i-\theta_j)+2\cos(\theta_i+\theta_j)+2\cos(\theta_k).$$
    Therefore, if $pq+r=2\cos(\varphi)$ and $\varphi$ is also a rational multiple of $\pi$, then we must have $\cos(\varphi)=\cos(\theta_i-\theta_j)+\cos(\theta_i+\theta_j)+\cos(\theta_k)$, which we may rewrite as the identity
    \begin{equation}\label{eq:angle_identity}
        \cos(\varphi)-\cos(\theta_i-\theta_j)-\cos(\theta_i+\theta_j)-\cos(\theta_k)=0.
    \end{equation}
    By \cite[Theorem 7]{TrigRationalAngle}, there are only finitely many possibilities for such a $\varphi$ to solve \Cref{eq:angle_identity}. The same argument for $|pq-r|$ likewise yields finitely many possibilities for $|pq-r|$ to be twice the cosine of a rational angle. That is, there are only finitely many triangles with three rational angles in which a reflection would yield another triangle with three rational angles. At the same time, \cite[Lemma 5.4]{FT2017} shows that any quiver with a geometric realization in $\mathbb{H}^2$ has an infinite mutation class. Hence, there must exist some triangle in the geometric realization in which not all of the angles are rational multiples of $\pi$. This reduces to case $(b)$ above and it follows that $[Q]$ is unbounded.
    \end{itemize}
\end{enumerate}

\noindent This conclude the alternative proof for \cref{thm:detbdd} using the geometric realizations of \cite{FT2017}.


\section{Final remarks and questions} A few final comments on the results and proofs in Sections \ref{sec:proof_mainresults} and \ref{sec:alternative_proof_mainresults} above:\\

\noindent (1) On the hypotheses of \cref{thm:detbdd}, it is not true in general that $C(Q)\leq4$ implies $[Q]$ is bounded. For example, the cyclic quiver $Q=(3,3,3)$ has Markov constant $C(3,3,3)=0\leq 4$ and it is mutation cyclic and unbounded. Hence, the additional requirement $p\leq 2$ is necessary.\\

\noindent (2) The simpler part of the proof of \cref{thm:main} established that $\sqrt{C(Q)}$ is an actual bound, but to prove boundedness, we need not give a particularly sharp one. Here is a more intuitive argument that a mutation acyclic class $[Q]$ with $C(Q)< 4$ must necessarily be bounded, which does nevertheless not provide such a sharp upper bound.\\

\noindent We consider the weights of a quiver $Q=(p,q,r)$ as a point $(p,q,r)\in\R^3$ and think of applying mutations as a continuous action of $G=\Z_2\ast\Z_2\ast\Z_2$ on $\R^3$. Specifically, here $G=\Z_2\ast\Z_2\ast\Z_2$ is the quotient of the free group $\mathbb{F}_3=\langle \mu_1,\mu_2,\mu_3\rangle$ by the relations $\mu_i^2=1$.\footnote{We are admittedly not allowing relabeling of vertices here, but that is irrelevant for the argument.} Its elements can therefore be written as sequences $\mu_{i_k}\cdots \mu_{i_2}\mu_{i_1}$, where $i_j\in{1,2,3}$, and we may assume them to be reduced, i.e.~no two consecutive $i_j,i_{j+1}$ are identical. The group operation on $G$ thus corresponds to concatenation of mutation sequences. In this context, we let $G$ act on $\R^3$ via the map $\mbox{mut}:G\times\R^3\longrightarrow\R^3$ given by
$$\mbox{mut}(\mu_{i_k}\cdots \mu_{i_2}\mu_{i_1};p(Q),q(Q),r(Q))=(p(\mu_{i_k}\cdots \mu_{i_2}\mu_{i_1}(Q)),q(\mu_{i_k}\cdots \mu_{i_2}\mu_{i_1}(Q)),r(\mu_{i_k}\cdots \mu_{i_2}\mu_{i_1}(Q))),$$

\noindent where $p(Q'),q(Q'),r(Q')$ denote the weights of a quiver $Q'$ at the corresponding arrows, again understood as a point in $\R^3$. For instance, if $Q=(p,q,r)$, then $\mbox{mut}(\mu_i;p,q,r)$ are the weights of $\mu_i(Q)$, $i\in[1,3]$. By the definition of quiver mutation, cf.~\cite[Chapter 2]{FWZ}, the weights change according to a composition of continuous functions, including $\max$, and thus the action is continuous.

Now, the constraint $C(Q)<4$ implies that the real weights of $Q$ must be in the subset
\begin{align*}
X&:=\{(x,y,z)\in\mathbb{R}^3:x^2+y^2+z^2+|xyz|\leq4\}\subset \mathbb{R}^3,\quad\mbox{ if $Q$ is acyclic,}\\
Y&:=\{(x,y,z)\in\mathbb{R}^3:x^2+y^2+z^2-|xyz|\leq 4\}\subset \mathbb{R}^3,\quad\mbox{ if $Q$ is cyclic.}
\end{align*}
Note that $X$ is bounded as a subset of $\R^3$. Thus, all acyclic quivers in $[Q]$ will remain in $X$, which is bounded. For the cyclic quivers, we note that $Y$ contains 8 singular points $S=\{(\pm2,\pm2,\pm2)\}$ such that $Y\setminus S$ has a unique bounded component. Thus, the fact that any acyclic quiver lies in such bounded connected component and continuity of the above $G$-action imply that any cyclic quiver must actually remain in that bounded component. This implies that $[Q]$ is bounded under these hypothesis.\\

\noindent (3) In line with the alternative proof above, there is another way to establish $\sqrt{C(Q)}$ as an upper bound, by using proof of \cite[Theorem 6.2]{FT2017}. Indeed, in this mutation-acyclic case with $C(Q)\in(0,4)$, \cite{FT2017} established that the geometric realization is a triple of lines in $S^2$, and the weights of the quiver are $2\cos(\theta)$, where $\theta$ is the angle between two lines in the geometric realization. It follows that smaller angles would correspond to larger weights. In the proof of \cite[Theorem 6.2]{FT2017}, it is shown that the angles in the realization of other quivers in the mutation class cannot be smaller than the angle
$$\theta=\arcsin \frac{\sqrt{4-C(Q)}}{2}.$$
By constructing the corresponding right triangle, we obtain the following figure:
\begin{center}
\begin{tikzpicture}[
  my angle/.style={
    every pic quotes/.append style={text=black},
    draw=black,
    angle radius=1cm,
  }]
  \coordinate (C) at (-1.5,-1);
  \coordinate (A) at (1.5,-1);
  \coordinate (B) at (1.5,1);
  \draw (C) -- node[above] {$2$} (B) -- node[right] {$\sqrt{4-C(Q)}$} (A) -- node[below] {$\sqrt{C(Q)}$} (C);
  \draw (A) +(-.25,0) |- +(0,.25);
  \pic [my angle, "$\theta$"] {angle=A--C--B};
\end{tikzpicture}
\end{center}
Therefore, any of the angles in the realization will never be smaller than such an angle $\theta$. In consequence, the weights of quivers in the mutation class will be bounded above by
$$2\cos(\theta)=2\cdot \frac{\sqrt{C(Q)}}{2}=\sqrt{C(Q)},$$
in line with \cref{thm:main}.(i).\\

\noindent (4) For bounded mutation classes, we established that $||Q||\leq \sqrt{C(Q)}$. In fact, we believe that this bound is sharp for infinite mutation classes. That is, one should be able to construct a sequence of mutations on $Q$ to obtain a quiver with one of the weights arbitrarily close to $\sqrt{C(Q)}$ and thus the remaining two weights arbitrarily close to $0$. Though there is additional work to be done in order to present a rigorous proof for sharpness, we now provide a heuristic strategy in favor of this upper bound being sharp.

\noindent First, note that in proving that a mutation class is unbounded for $C(Q)>4$, we applied mutations on cyclic quivers to obtain quivers with increasing weights, cf.~\cref{lem:mu_*}. That said, in the case of $C(Q)<4$, the same mutations will no longer guarantee increasing weights.  
Now, if we have a cyclic quiver, then we can apply mutations to attain an acyclic quiver and then apply a mutation to increase the largest weight. Indeed, let us assume that we have a cyclic quiver $Q_0$ with weights $p,q_0,r_0$ and $0<r_0\leq q_0\leq p\leq C(Q)$, depicted as follows:
\begin{center}
    \begin{tikzcd}[column sep=1cm, row sep=1cm]
Q_0 & 2 \ar[rdd, "q_0"]\\
&\\
1 \ar[ruu, "p"] && 3\ar[ll,"r_0"]\\
\end{tikzcd}
\end{center}
\vspace{-10mm}
By \cite[Lemma 6.1]{FT2017}, there exists an acyclic representative in the mutation class that has a weight $p$. This is done by alternating mutations at vertices $1$ and $2$, so as to keep the weight $p$ fixed. Thus, we can and do assume now that we have an acyclic quiver $Q$ with weights $p,q,r$. There are six possible orientations for such a quiver, which we list as follows:
\begin{center}
\begin{tikzcd}[column sep=1cm, row sep=1cm]
& 2\\
& \\
1 \ar[ruu,"p"] \ar[rr, "r"] && 3 \ar[luu, "q"] \\
\end{tikzcd}
$\xleftrightarrow{\text{\hspace{5mm}$\mu_2$}\hspace{5mm}}$
    \begin{tikzcd}[column sep=1cm, row sep=1cm]
& 2 \ar[ldd, "p"] \ar[rdd, "q"]\\
&\\
1 \ar[rr, "r"] && 3 \\
\end{tikzcd}
$\xleftrightarrow{\text{\hspace{5mm}$\mu_3$}\hspace{5mm}}$
\begin{tikzcd}[column sep=1cm, row sep=1cm]
& 2 \ar[ldd,"p"] \\
& \\
1 && 3 \ar[ll, "r"] \ar[luu, "q"] \\
\end{tikzcd}
\end{center}
\vspace{-10mm}
\begin{center}
\begin{tikzcd}[column sep=1cm, row sep=1cm]
& 2 \ar[ldd, "p"] \ar[rdd, "q"]\\
& \\
1 && 3 \ar[ll, "r"] \\
\end{tikzcd}
$\xleftrightarrow{\text{\hspace{5mm}$\mu_2$}\hspace{5mm}}$
    \begin{tikzcd}[column sep=1cm, row sep=1cm]
& 2 \\
&\\
1 \ar[ruu, "p"] && 3 \ar[ll, "r"] \ar[luu, "q"] \\
\end{tikzcd}
$\xleftrightarrow{\text{\hspace{5mm}$\mu_3$}\hspace{5mm}}$
\begin{tikzcd}[column sep=1cm, row sep=1cm]
& 2 \ar[rdd,"q"] \\
& \\
1 \ar[ruu, "p"] \ar[rr, "r"] && 3 \\
\end{tikzcd}
\end{center}
\vspace{-10mm}
Let us assume the given quiver $Q$ is one of the two leftmost quivers, since the remaining four are mutation equivalent to one of these. In this case, a mutation at vertex $3$ yields a cyclic quiver with weights $p',q,r$ with $p'=p+qr$. By this process, done iteratively, we obtain strictly increasing largest weights for the quivers. Then, we can apply alternating mutations between vertices $1$ and $2$ to obtain another acyclic quiver $Q'$ with weight $p',q',r'$. Noting that $p'=p+qr$, we deduce the following identities by setting the Markov constants of $Q,Q'$ equal to each other:
\begin{align*}
    (p')^2+(q')^2+(r')^2+p'q'r'&=p^2+q^2+r^2+pqr&\Longleftrightarrow\\
    p'(p+qr)+(q')^2+(r')^2+p'q'r'&=p^2+q^2+r^2+pqr&\Longleftrightarrow\\
    (q')^2+(r')^2+p'q'r'+p'qr&=p^2+q^2+r^2+pqr-p'p&\Longleftrightarrow\\
    (q')^2+(r')^2+p'q'r'+p'qr&=p^2+q^2+r^2+pqr-p^2-pqr=q^2+r^2.
\end{align*}
Therefore we obtain the inequality $(q')^2+(r')^2<q^2+r^2$. By repeating this process infinitely many times, we expect to obtain quivers with largest weight $p$ arbitrarily close to $\sqrt{C(Q)}$, while the other two weights $q,r$ become arbitrarily close to $0$.\\

\noindent (5) As obvious as it is, we record this goal: {\it Find a complete and verifiable characterization of bounded mutation classes for quivers of higher rank.} That is, prove a result such as \cref{thm:main} and \cref{thm:detbdd} (or better) for quivers of rank $4$ and above.\\

\noindent (6) It would be interesting to understand the dynamical properties of quiver mutations for real weights in rank 3 and beyond. That is, understanding the weights of quivers as points in $\R^{{n\choose 2}}$, to characterize the distribution of points of a given orbit, either quantitatively or qualitatively. For instance, if the orbit is bounded, understanding what are the possible limit sets, e.g.~whether they are dense in some positive-dimensional subset of $\R^{{n\choose 2}}$. If the orbit is unbounded, it would be interesting to establish quantitative growth estimates, and study whether there are directions or cones of divergence.

\noindent Note that already in rank 3, it would be interesting to understand properties of mutation orbits beyond boundedness, e.g.~density in the given two level sets (cyclic and acyclic) for a fixed value of the Markov constant, ergodicity, or any other type of measure-related properties, see e.g.~Figures \ref{fig:Examples1} and \ref{fig:Examples2} above. The code in \cref{ssec:code} can maybe be helpful to develop intuition.

\section{Appendix: SageMath Code and details for figures}\label{sec:appendix}

The following two subsections contain some of the experimental data displayed in the introduction and SageMath code that produces many such examples. These examples and the code are {\it not} logically needed for the mathematical results of the article. That said, we found running this code helpful to gain intuition on the dynamics of real quiver mutation, in particular towards guessing the statement of what ended up being \cref{thm:main} and the key inequality in Equation \ref{eq:Markov_inequality}.

\subsection{SageMath code to generate and plot random mutation sequences of a quiver}\label{ssec:code} The main line of commands to produce a random sequence of mutations and their plot uses the following series of functions. The first function inputs a skew-symmetric matrix $B\in M_3(\Q)$ and an index $k\in\{1,2,3\}$ and outputs the mutated exchange matrix $\mu_k(B)$, as follows:

\tiny
\begin{lstlisting}[language=Python]
def mutate_matrix_rational_mpl(B, k):
 
 if B.dimensions() != (3, 3):
    raise ValueError("Input matrix B must be a 3x3 matrix.")
 if not B ==-B.transpose():
    raise ValueError("Input matrix B must be skew-symmetric.")
 if k not in [1, 2, 3]:
    raise ValueError("Mutation index k must be 1, 2, or 3.")
 B_prime = matrix(QQ, 3, 3)
 for i in range(3):
    for j in range(3):
        if i == k- 1 or j == k- 1:
            B_prime[i, j] =-B[i, j]
        else:
            bik = B[i, k- 1]
            kj = B[k- 1, j]
            if bik > 0 and kj > 0:
                B_prime[i, j] = B[i, j] + bik * kj
            elif bik < 0 and kj < 0:
                B_prime[i, j] = B[i, j]- bik * kj
            else:
                B_prime[i, j] = B[i, j]
return B_prime

\end{lstlisting}
\normalsize
\noindent This next function inputs a skew-symmetric matrix $B\in M_3(\Q)$ and a sequence of indices $(i_1,\ldots,i_\ell)$ with $i_j\in\{1,2,3\}$. It applies this sequence of mutations to $B$ and plots the resulting entries $(1,2)$, $(2,3)$ and $(1,3)$ of the mutated exchange matrices $(\mu_{i_j}\circ\ldots\circ\mu_{i_1})(B)$. The output is the list of such coordinate entries for the mutate matrices and the plot:
\tiny
\begin{lstlisting}[language=Python]
def apply_mutation_sequence_plot_mpl(B_initial, mutation_sequence):

    mutated_data = []
    B_current = B_initial
    mutated_data.append((float(B_current[0, 1]), float(B_current[1, 2]),float(B_current[0, 2])))
 
    for k in mutation_sequence:
        B_current = mutate_matrix_rational_mpl(B_current, k)
        mutated_data.append((float(B_current[0, 1]), float(B_current[1, 2]),
        $\sqcup\xhookrightarrow$float(B_current[0, 2])))
 
    fig = plt.figure()
    ax = fig.add_subplot(111, projection='3d')
    x = [data[0] for data in mutated_data]
    y = [data[1] for data in mutated_data]
    z = [data[2] for data in mutated_data]
    ax.scatter(x, y, z, c='blue', marker='o')
    ax.set_xlabel('b12')
    ax.set_ylabel('b23')
    ax.set_zlabel('b13')
    ax.set_title('Mutation Sequence in (b12, b23, b13) Space')
return mutated_data, plt

\end{lstlisting}
\normalsize
\noindent The following function inputs a length $\ell\in\N$ and outputs a randomly generated sequence of indices $(i_1,\ldots,i_\ell)$ with $i_j\in\{1,2,3\}$, of length $\ell$ with no two consecutive indices being equal, i.e.~$i_j\neq i_{j+1}$:
\footnotesize
\begin{lstlisting}[language=Python]
import random
def generate_alternating_sequence(length):
 
     if length <= 0:
       return []
    sequence = []
    first_number = random.choice([1, 2, 3])
    sequence.append(first_number)
 
    for _ in range(length- 1):
        possible_next = [num for num in [1, 2, 3] if num != sequence[-1]]
        next_number = random.choice(possible_next)
        sequence.append(next_number)
return sequence

\end{lstlisting}
\normalsize
\noindent This next function inputs a vector $v=(x,y,z)\in \Q^3$ and outputs the skew-symmetric matrix $M(v)\in M_3(\Q)$ with entries $(1,2)$ being $x$, $(2,3)$ being $y$ and $(1,3)$ being $z$:
\footnotesize
\begin{lstlisting}[language=Python]
def skew_symmetric_matrix(v):

     x, y, z = v
     M = matrix(QQ, 3, 3)
     M[0, 1] = x
     M[1, 0] =-x
     M[1, 2] = y
     M[2, 1] =-y
     M[0, 2] = z
     M[2, 0] =-z
return M
\end{lstlisting}
\normalsize
\noindent The {\bf main line of commands} that the user can choose to execute is as follows:
\footnotesize
\begin{lstlisting}[language=Python]
vector = (-0.6,-0.43,0.567)
B_rational_mpl = skew_symmetric_matrix(vector)
print("Initial Exchange Matrix is:")
show(B_rational_mpl)
length1 = 100
mutation_seq_mpl = generate_alternating_sequence(length1)
print(f"Random alternating sequence of length {length1}: {mutation_seq_mpl}")
mutation_history_mpl,plot_mpl=apply_mutation_sequence_plot_mpl(B_rational_mpl,mutation_seq_mpl)
plot_mpl.savefig("FigurePlot3D.pdf")\end{lstlisting}
\normalsize
\noindent In the line of commands above, the user chooses the quiver $Q=(p,q,r)$ by selecting the tuple ``vector''. The user also selects the length $\ell\in\N$ of the mutation sequence by choosing the value of ``length1''. The plot is then saved in the document named ``FigurePlot3D.pdf''.


\subsection{Mutation sequences for Figures \ref{fig:Examples1} and \ref{fig:Examples2}}\label{ssec:explicit_seq} Figures \ref{fig:Examples1} and \ref{fig:Examples2} have been generated by the code in \cref{ssec:code}. This subsection displays the precise mutation sequences $(i_1,\ldots,i_\ell)$ used in each of the examples. Here $\ell$ is the length of the mutation sequence and the indices are always chosen such that $i_j\neq i_{j+1}$ for all indices, to avoid involutive steps in the sequence. To start, for the quiver $Q=(-0.02,-0.01,0.03)$ in Figure \ref{fig:Examples1} (left), the mutation sequence $(i_1,\ldots,i_\ell)$ has $\ell=150$ and reads
\tiny
$$(2,3,2,1,3,2,1,3,1,2,1,2,
3,2, 1,3, 1, 2, 3,1,2,1,2,1,2,3,1,2,1,3,1,3,2,3,1,3,2,1, 2,
1,2, 1,2, 1, 3, 1,2,1,2,3,2,$$ $$1,3,2,3,1,2,3,1,2,1,2,3,1,3, 2,
1,2, 1,3, 1, 3, 2,3,2,3,2,3,2,3,2,1,3,1,3,1,2,1,2,3,1,2, 1,
3,1, 2,1, 2, 3, 2,1,3,2,1,$$
$$3,2,1,2,1,
3,1,2,3,2,1,2,3,2,1, 2,
1,3, 2,3, 2, 3, 2,3,1,3,2,1,3,2,1,2,3,2,3,2,1,2,3,1,3,2, 3,
1,2, 1).$$
\normalsize
For the quiver $Q=(-0.9,-0.22,0.7106)$ in Figure \ref{fig:Examples1} (center), $\ell=100$ and the sequence is
\tiny
$$(2,3,1,3,1,3,2,1,3,1,3,1,
2,3, 1,2, 1, 2, 1,2,1,3,1,2,1,2,3,2,3,1,3,2,1,3,2,1,3,2, 3,
1,3, 2,1, 2,3, 2,1,3,2,$$
$$3,1,2,3,1,2,3,2,1,2,1,3,1,2,1,2, 3,
2,1, 3,2, 1, 2, 3,2,3,1,2,3,1,3,2,3,1,2,1,3,1,3,1,3,1,3, 1,
2,3, 2,1, 2, 3, 1).$$

\normalsize
\noindent For the quiver $Q=(-0.84,-0.26,0.11)$ in Figure \ref{fig:Examples1} (right), $\ell=125$ and the sequence is
\tiny

$$(3,2,3,1,2,3,2,1,2,1,2,3,
2,1, 3,2, 3, 2, 3,1,2,1,3,2,1,2,3,2,3,1,3,1,3,1,3,2,1,3, 2,
1,3, 1,$$
$$2, 3, 1, 2,1,3,1,2,3,2,1,2,1,3,1,2,1,3,2,3,2,1,2, 3,
1,3, 1,3, 2, 1, 2,3,1,3,1,3,2,1,3,2,1,2,$$
$$3,1,3,2,3,2,1,2, 1,
3,1, 2,1, 3, 1, 3,1,3,2,1,2,3,1,3,2,3,2,3,2,1,2,3,1,2,3, 2,
1,3, 2,3, 1).$$
\normalsize

\noindent Figure \ref{fig:Examples2} plots quivers in the mutation class of the quiver $Q=(-0.6,-0.43,0.567)$. Specifically, it plots the images of $Q$ after applying three different sequences of mutations. The sequence are described as follows. For Figure \ref{fig:Examples2} (right), the mutation sequence has length $\ell=100$ and it is
\tiny
$$(3,1,2,1,2,1,3,2,3,1,2,1,
2,1, 3,1, 3, 1, 2,1,3,2,1,3,2,1,3,2,1,3,1,3,1,3,1,2,1,2, 3,
2,1, 3,2, 1, 3, 2,1,3,2,1,$$
$$3,1,2,1,3,2,1,3,1,3,2,3,2,1,2, 1,
3,1, 2,3, 2, 1, 3,1,3,2,3,2,3,1,2,3,1,2,3,2,3,1,2,3,2,1, 3,
2,3, 2,1, 2, 3, 1).$$
\normalsize

\noindent For Figure \ref{fig:Examples2} (center), $\ell=100$ and the mutation sequence is
\tiny
$$(1,2,1,3,2,3,2,1,3,2,1,2,
1,3, 1,3, 2, 1, 3,1,2,3,1,2,3,2,3,1,2,3,1,2,1,3,1,2,1,2, 1,
3,2, 1,2, 3, 1, 2,3,1,3,1,$$
$$2,3,2,3,1,3,2,3,1,3,1,2,3,1,2, 3,
2,3, 1,3, 2, 1, 2,1,3,1,2,1,3,1,3,2,1,3,1,3,2,3,2,1,2,1, 3,
2,3, 2,1, 3, 1, 2).$$
\normalsize

\noindent For Figure \ref{fig:Examples2} (right), $\ell=125$ and the mutation sequence is
\tiny
$$(2,3,2,3,2,1,2,1,2,3,2,3,
1,3, 1,2, 1, 3, 1,3,1,3,1,3,2,1,2,3,1,3,1,3,2,1,3,2,1,2, 3,
1,3,$$
$$ 2,3,2, 1, 3,2,1,3,2,3,1,3,1,2,1,3,2,3,1,3,1,3,2,1, 3,
2,3, 1,3, 2, 1, 2,1,2,3,2,3,1,3,2,1,2,$$
$$3,1,3,2,3,2,3,1,3, 1,
2,3, 1,3, 1, 3, 2,1,2,3,1,2,1,2,1,3,2,3,1,3,2,3,1,3,1,2, 1,
3,1, 2,1, 3).$$
\normalsize

\noindent The reader can generate many other such figures by running the code in the Subsection \ref{ssec:code}.

\bibliographystyle{plain}
\bibliography{ref}

\end{document}